\documentclass[10pt,a4paper]{article}
\usepackage[utf8]{inputenc}
\usepackage[english]{babel}
\usepackage{amsmath}
\usepackage{amsfonts}
\usepackage{amssymb}
\usepackage{amsthm}
\usepackage{graphicx}

\newtheorem{theorem}{Theorem}[section]
\newtheorem{thm}[theorem]{Theorem}

\newtheorem{cor}[theorem]{Corollary} 
\newtheorem{definition}[theorem]{Definition} 
\newtheorem{lem}[theorem]{Lemma}

\theoremstyle{remark}
\newtheorem{example}[theorem]{Example}
\newtheorem{rem}[theorem]{Remark}

\numberwithin{claim}{theorem}

\begin{document}

 \title{On generalized constant ratio surfaces with higher codimension}

\author{Alev Kelleci\footnote{F\i rat University, Faculty of Science, Department of Mathematics, 23200 Elaz\i\u g, Turkey. e-mail: alevkelleci@hotmail.com},  Nurettin Cenk Turgay \footnote{Istanbul Technical University, Faculty of Science and Letters, Department of  Mathematics, 34469 Maslak, Istanbul, Turkey. e-mail: turgayn@itu.edu.tr} and Mahmut Erg\" ut\footnote{Nam\i k Kemal University, Faculty of Science and Letters, Department of Mathematics, 59030 Tekirda\u g, Turkey. email: mergut@nku.edu.tr}}

\maketitle

\begin{abstract}
In this paper, we study generalized constant ratio surfaces in the Euclidean 4-space. We also obtain a classifications of constant slope surfaces.


\noindent \textbf{Keywords.}
\end{abstract}

\section{Introduction}\label{sec:1}
 
It is well-known that the position function is the simplest  geometrical object
associated with a submanifold in a Euclidean space. Because of this reason, 
many problems related with understanding geometry of submanifolds with a given 
restriction on their position vectors have been studied by many mathematicians 
so far.

In this direction, the notion of constant ratio submanifolds in Euclidean spaces
 firstly introduced by B.-Y. Chen in \cite{ChenCR}. Let $M$ be a submanifold 
isometrically immersed in $\mathbb E^m$, there is a natural orthogonal 
decomposition of the position vector $x$ at each point on M; namely
\begin{equation}\label{DeCompofx}
x=x^T+x^{\bot}
\end{equation}
where $x^T$ and $x^{\bot}$ denote the tangential and normal components of
 $x$, respectively. A submanifold of Euclidean space is said to be of \textit%
{constant ratio} if the ratio of the lengths of the tangential and normal 
components of its position vector is constant. Complete classification of
 constant ratio hypersurfaces in Euclidean spaces was obtained in 
\cite{ChenCR} (see also \cite{B2007}). In addition, constant ratio space-like
 submanifolds in pseudo-Euclidean space have been completely 
classified in \cite{Chen2003}. On the other hand, the submanifold $M$ is said
 to be  a $T$-constant (respectively, $N$-constant) submanifold if $x^T$ 
(respectively, $x^\bot$) has constant length. B. Y. Chen studied $T$-constant 
and $N$-constant submanifolds in semi-Euclidean spaces in \cite{ChenTNsub}.

Another important class of submanifolds is constant slope (CS) hypersurfaces 
defined by M. I. Munteanu in \cite{Munt2010}. When the codimension of $M$ is 1,
 it is said to be a CS hypersurface if its position vector $x$  makes a constant
 angle with its unit normal vector field, $N$. In this case $N$ is parallel, 
one can generalize this definition to higher codimensions as following:
\begin{definition}\label{DefoFCSS}
Let $M^n$ be a submanifold in a semi-Euclidean space $\mathbb E^m$ with the 
position vector $x$. $M$ is said to be  a constant slope submanifold if there 
exists a parallel normal vector field $N$ which makes constant angle with $x$.
\end{definition}
In \cite{Munt2010}, constant slope surfaces (CSS) in $\mathbb E^3$ have been studied
 and this study moved into the Minkowski 3-space in \cite{FY2012,FW2013}. First of two
 main purposes of this article is to study constant slope surfaces in the Euclidean 
4-space $\mathbb E^4$.

Constant slope (CS) hypersurfaces posses an interesting property: The tangential component $x^T$ of the position vector
 onto the tangent plane of the surface $M$ is a principal direction (see \cite{Munt2010}). As we will describe in Sect. 3, the tangential component of the position vector of a CSS is  a principal direction of all of its shape operators (See Corollary \ref{COROCSS1}). As a generalization of the concept of constant slope surfaces, in \cite{FM2014}, Fu and Munteanu   studied  surfaces in the Euclidean 3-space with the property that the tangential component of the position vector remains a principal direction but without the restriction of being constant of the angle function. They preferred to use generalized constant ratio surfaces (shortly, GCR surfaces), 
in order to point out the connection with the CR surfaces defined by Chen (See Sect. 4).   Recently in \cite{FuYang2016,YuFu2017GCRS} the completed classification of GCR surfaces in Minkowski 3-spaces has been obtained. Also in \cite{FM2014} and \cite{FuYang2016,YuFu2017GCRS}, all flat and constant mean curvature GCR surfaces, respectively, in $\mathbb E^3$ and 
$\mathbb E^3_1$ were classified.

Before we proceed to the other definition that we would like to give, we want to mention about class $\mathcal A$ immersions into the product spaces  $Q^n_c \times \mathbb R$, where $Q^n_c$ denotes the Riemannian space form with dimension $n$ and sectional curvatures $c=\pm1$. An immersion $f:M^n\rightarrow Q^n_c \times \mathbb R$ is said to belong class $\mathcal A$
if the tangential part of $\partial_t$ is principal directions of all shape operators of $f$, \cite{MT}. The notion of $\mathcal A$ immersions generalize constant principle direction (CPD) surfaces in $Q^2_c \times \mathbb R$ into higher codimesions 
(See \cite{DFJvK2009,DMuntNistr2011,To,AMT2017} for CPD surfaces). By using a similar idea, one can give the definition of GCR submanifolds as following: 
\begin{definition}\label{DefoFGCR}
Let $M^n$ be a submanifold in $\mathbb{E}^{m}$ and $x$ be its position vector.
 $M$ is said to be a generalized constant ratio (GCR) submanifold if the 
tangential part $x^T$ of $x$ is one of principal directions of all shape 
operators of $M$.
\end{definition}
As we mention in Sect. 4, GCR submanifolds defined as above satisfy some geometrical 
properties being similar to GCR surfaces in 3-dimesional semi-Euclidean spaces. Furthermore,
 we would like to add that R. Tojiero and B. Mendoza the definition of the class$-\mathcal A$
 of the spaces is used in the case of codimension being larger than one. By using this 
definition given in \cite{MT,To}, a submanifold of Euclidean and semi-Euclidean spaces 
is called a GCR submanifold if the tangential component of the position vector of that 
submanifold is a principal direction of all shape operators. The other aim of this 
article is to understand GCR surfaces in the Euclidean 4-space $\mathbb E^4$.

This paper is organized as follows. In Sect. 2, we introduce the notation that 
we will use and give a brief summary of basic definitons in theory of submanifolds 
in Euclidean spaces. In Sect. 3, we obtain the complete classification of CSS in 
the Euclidean 4-space. In last section, we obtain the complete classification of GCR 
surfaces in the Euclidean 4-space.

\section{Basic notation and definitions}
Let $\mathbb E^m$ denote the Euclidean $m$-space with the canonical Euclidean metric tensor given by  
$$
\widetilde g=\langle\ ,\ \rangle=\sum\limits_{i=1}^m dx_i^2,
$$
where $(x_1, x_2, \hdots, x_m)$  is a rectangular coordinate system in $\mathbb E^m$. 

Consider an $n$-dimensional Riemannian submanifold  of the space $\mathbb E^m$. We denote Levi-Civita connections of $\mathbb E^m$ and $M$ by $\widetilde{\nabla}$ and $\nabla$, respectively. The Gauss and Weingarten formulas are given, respectively, by
\begin{eqnarray}
\label{MEtomGauss} \widetilde\nabla_X Y&=& \nabla_X Y + h(X,Y),\\
\label{MEtomWeingarten} \widetilde\nabla_X \xi&=& -S_{\xi}(X)+D_X {\xi},  
\end{eqnarray}
whenever $X,Y$ are tangent and $\xi$ is normal vector field on $M$, where $h$,  $D$  and  $S$ are the second fundamental form, the normal connection and  the shape operator of $M$, respectively. It is well-known that the shape operator and the second fundamental form are related by  
$$\left\langle h(X, Y), {\xi} \right\rangle = \left\langle S_{\xi}X, Y \right\rangle.$$

The Gauss and Codazzi equations are given, respectively, by
\begin{eqnarray}
\label{MinkGaussEquation} \langle R(X,Y)Z,W\rangle&=&\langle h(Y,Z),h(X,W)\rangle-
\langle h(X,Z),h(Y,W)\rangle,\\
\label{MinkRicci} \langle R^D(X,Y)\xi,\eta \rangle&=&\langle \lbrack S_\xi,S_\eta]X,Y\rangle, \\
\label{MinkCodazzi} (\nabla_X h )(Y,Z)&=&(\nabla_Y h )(X,Z),
\end{eqnarray}
whenever $X,Y,Z,W$ are tangent to $M$, where $R,\; R^D$ are the curvature tensors associated with connections $\nabla$ 
and $D$, respectively. We note that  $\bar \nabla h$ is defined by
$$(\nabla_X h)(Y,Z)=D_X h(Y,Z)-h(\nabla_X Y,Z)-h(Y,\nabla_X Z).$$ 
A submanifold  $M$ is said to have flat normal bundle if $R^D=0$  identically.

The mean curvature vector field $H$ of the surface $M$ is defined as 
\begin{equation} \label{MC}
H=\frac{1}{2}tr h.
\end{equation}\label{GC}
If $M$ is a surface, i.e, $n=2$, then the Gaussian curvature $K$ of the surface $M^2$ is defined as 
\begin{equation}
K=\frac{R(X,Y,X,Y)}{Q(X,Y)},
\end{equation}
if $X$ and $Y$ are chosen so that $Q(X,Y)=\langle X,X\rangle\langle Y,Y\rangle-\langle X,Y\rangle^2$ does not vanish.

\section{Classifications of Constant Slope Surfaces in $\mathbb E^4$} \label{S:Classification0}

In this section, we would like to study constant slope surfaces in the Euclidean $4$-space. First, we would like to present an example of CSS in $\mathbb E^4$.

\begin{example}\label{CSSExampleE4}
Let $\Pi$ be a hyperplane $\mathbb E^4$ with the unit normal   $c_0$. Assume that $c_2=c_2(t)$ be a curve lying on ${\mathbb S^2(r^2)}={\mathbb S^3(1)}\cap \Pi$ such that $0 < r \leq1$. Then, the surface $M$ is given by
\begin{equation} \label{x0}
x(s,t)=s\cos \left(\tan \theta \ln s\right) c_0+s \sin \left(\tan \theta \ln s\right)  c_2(t)
\end{equation}
for a constant $\theta$. By a direct computation, one can see that $x$ can be decomposed as 
\begin{equation} \label{gcrx0}
x(s,t)=s \cos^2 \theta \partial_s+s \sin \theta e_3
\end{equation}
for the unit normal vector field $e_3$ given by
\begin{equation} \label{e30}
e_3 = \sin(\tilde{u})c_0-\cos(\tilde{u})c_2(t),
\end{equation}
where $theta=\hat{(x,e_3)}$ and $\tilde{u}(s)=\theta+u(s)$. Furthermore, $e_3$ is parallel. Hence, $M$ is CSS.
\end{example}

Now, assume that $M$ is a constant slope surface  in $\mathbb E^4$ and let $x:\Omega\rightarrow M$ be its position vector, where $\Omega$ is an open subset of $\mathbb R^2$. We define a non-negative, smooth function $\mu$ by $\mu^2=\left\langle x,x\right\rangle$. Let ${e_3}$ be the unit parallel vector field such that $\langle x,e_3\rangle=\sin\theta$, for a constant $\theta$. In this case, \eqref{DeCompofx} turns into
\begin{equation} \label{ExpofS1000}
x= \mu\cos \theta e_1+ \mu\sin \theta {e_3},
\end{equation}
for a unit tangent vector field $e_1$. Let $e_2$ and ${e_4}$ be a unit vector field and a unit normal vector field, satisfying $\langle e_1,e_2\rangle=0$ and $\left\langle {e_3},{e_4}\right\rangle=0$, respectively. Since the codimension is 2, $e_4$ is also parallel. Moreover, having parallel frame field of the normal space of $M$ implies $R^D=0$ (See \cite{ChenGS}). 

\begin{rem}
If $M$ is  a surface lying on a hyperplane $\Pi$ of $\mathbb E^4$, then $M$ is obviously a CSS with $\theta=0$ and in this case $e_3$ can be taken as the unit normal of $\Pi$. On the other hand, if $\Pi$ lies on a sphere $\mathbb S^3(r^{-2})$ of radius $r$, then $M$ is again a CSS. However, this time we have $\theta=\pi/2$  and $e_3=x/r$. One can easily observe the converse of these facts: in the case $\theta=0$ and $\theta=\pi/2$ also implies $x(\Omega)\subset\Pi$ and $x(\Omega)\subset\mathbb S^3(r^{-2})$, respectively. In the remaining of this section, we exclude these trivial cases and assume $\theta\in(0,\pi/2)$.
\end{rem}

We note that the definition of $\mu$ directly implies
\begin{eqnarray}
\label{mus0} e_1(\mu)&=&\cos \theta, \\
\label{mut0} e_2(\mu)&=&0.
\end{eqnarray}
By a further computation considering \eqref{ExpofS1000} and the last two equations, we obtain the following lemma, where we put $S_{e_\beta}=S_\beta$ and $h^\beta_{ij}=\langle h(e_i,e_j),e_\beta\rangle.$

\begin{lem}\label{Case1ClassThmDiagonSpacelikekClm10}
The Levi-Civita connection $\nabla$ and the second fundemental form of $M$ are given by
\begin{subequations} \label{CASEISpacelikeLeviCivitaEq1ALL0}
\begin{eqnarray}
\label{CASEISpacelikeLeviCivitaEq1a0}\nabla _{e_{1}}e_{1}=0,  & & \nabla _{e_{1}}e_{2}=0, \\
\label{CASEILeviCivitaEq1b0} \nabla _{e_{2}}e_{1}= \frac{1+\mu h_{22}^3 \sin \theta}{\mu \cos \theta}e_2, & &\nabla _{e_{2}}e_{2}=-\frac{1+\mu h_{22}^3 \sin \theta}{\mu \cos \theta}e_1,\\
\label{secondfund0}h(e_1,e_1)=-\Big(\frac{\sin \theta}{\mu}\Big)e_3,\  h(e_1,e_2)=0,& & h(e_2,e_2)=h_{22}^3e_3+h_{22}^4e_4.
\end{eqnarray}
\end{subequations}
and the matrix representations of shape operator $S$ of $M$ with respect to $\{e_1,e_2\}$ are
\begin{align}\label{CASEISpacelikeLeviCivitaEq10}
\begin{split}
S_3=\left(\begin{array}{cc}
-\Big(\frac{\sin \theta}{\mu}\Big)&0\\
0&h_{22}^3
\end{array}\right), \quad
&S_4=\left(\begin{array}{cc}
0&0\\
0&h_{22}^4
\end{array}\right)
\end{split}
\end{align}
and $h_{22}^4$, $h_{22}^3$ satisfy 
\begin{subequations} \label{ClassThmDiagonSpacelikekCod1Case1ALL0}
\begin{eqnarray} \label{ClassThmDiagonSpacelikekCod1Case11a0}
e_1(h_{22}^3)&=&\frac{1+\mu h_{22}^3 \sin \theta}{\mu \cos \theta}(h_{11}^3-h_{22}^3),
\\\label{ClassThmDiagonSpacelikekCod1Case11b0}
e_1(h_{22}^4)&=&-\frac{1+\mu h_{22}^3 \sin \theta}{\mu \cos \theta} h_{22}^4,
\end{eqnarray}
\end{subequations}
\end{lem}

\begin{proof}
By considering \eqref{ExpofS1000}, $\tilde{\nabla}_X x=X$ and the normal vector field $e_3$ being parallel, one can get
\begin{equation}\label{ApplyXExpofS1000}
X=X(\mu)\cos \theta e_1+\mu \cos \theta \Big(\nabla_{X}e_1 + h(e_1,X)\Big)+X(\mu) \sin \theta{e_3}-\mu \sin \theta S_{3}X
\end{equation}
whenever $X$ is tangent to $M$. By considering \eqref{mus0} and taking $X=e_1$ in \eqref{ApplyXExpofS1000}, we obtain \eqref{CASEISpacelikeLeviCivitaEq1a0} together with
\begin{eqnarray}
\label{ApplyXExpofS10000Eq2b0} h_{11}^3= -\Big(\frac{\sin \theta}{\mu}\Big),\\
\nonumber h_{11}^4=0
\end{eqnarray}
which yields \eqref{CASEISpacelikeLeviCivitaEq10}. While considering \eqref{mut0}, \eqref{ApplyXExpofS1000} for $X=e_2$ gives 
\begin{equation*}
\nabla_{e_2}e_1= \frac{1+\mu h_{22}^3 \sin \theta}{\mu \cos \theta}e_2, \quad \nabla_{e_2}e_1= -\frac{1+\mu h_{22}^3 \sin \theta}{\mu \cos \theta}e_1, \quad h_{12}^3=0, \quad h_{12}^4=0.
\end{equation*} 
Thus, we have \eqref{CASEILeviCivitaEq1b0} and \eqref{secondfund0}. \eqref{ClassThmDiagonSpacelikekCod1Case11a0} and \eqref{ClassThmDiagonSpacelikekCod1Case11b0} follow from the Codazzi equation \eqref{MinkCodazzi} for $X=e_1$, $Y=Z=e_2$.
\end{proof}
We immediately have the following observation
\begin{cor}\label{COROCSS1}
If $M$ is a CSS  in $\mathbb E^4$, then the tangential component of its shape operator 
is one of principal directions of all shape operators of $M$.
\end{cor}

We will use the following lemma.

\begin{lem}\label{Case1ClassThmDiagonSpacelikekClm120}
There exists a local coordinate system  $(s,t)$ defined in a neighborhood $\mathcal N_p$ of a point $p\in M$ at which $\displaystyle \theta \notin \left\{0,\frac{\pi}{2}\right\}$ such that the induced metric of $M$ is
\begin{equation}\label{ClassThmDiagonSpacelikekDefgEqRESCase10}
g=\frac{1}{{\cos^2 \theta}}ds^2+m^2dt^2
\end{equation}
for a smooth function $m$ satisfying
\begin{equation}\label{Case1ClassThmDiagonSpacelikekDefm0}
e_1(m)-\frac{1+\mu h_{22}^3 \sin \theta}{\mu \cos \theta}m=0.
\end{equation}
Furthermore, the vector fields $e_1,e_2$ described as above become $\displaystyle e_1={\cos \theta}\partial_s$, \linebreak $\displaystyle e_2=\frac 1{m}\partial_t$ in $\mathcal N_p$.
\end{lem}

\begin{proof}
We have $\displaystyle [e_1,e_2]=-\frac{1+\mu h_{22}^3 \sin \theta}{\mu \cos \theta}e_2$ because of \eqref{CASEISpacelikeLeviCivitaEq1ALL0}. Thus, if $m$ is a non-vanishing smooth function on $M$ satisfying  \eqref{Case1ClassThmDiagonSpacelikekDefm0}, then we have $\displaystyle \left[\frac{1}{\cos \theta}e_1,me_2\right]=0$. Therefore, there exists a local coordinate system $(s,t)$ such that $e_1={\cos \theta}\partial_s$ and  $\displaystyle e_2=\frac 1m\partial_t$. Thus, the induced metric of  $M$ is given as in \eqref{ClassThmDiagonSpacelikekDefgEqRESCase10}.
\end{proof}
Now, we are ready to obtain the classification theorem.
\begin{thm}\label{ClassThmDiagonSpacelikek0}
A constant slope surface $M$  in $\mathbb E^4$ is locally congruent to the surface described in Example \ref{CSSExampleE4}
\end{thm}

\begin{proof}
In order to proof the necessary condition, we assume that $M$ is an oriented CSS with the isometric immersion $x:M\rightarrow \mathbb E^4$ satisfying $\left\langle x,x\right\rangle=\mu^2$. Since $M$ is a CSS, $x$ can be decomposed as given in \eqref{ExpofS1000}. Let $\{e_1,e_2;e_3,e_4\}$ be the local orthonormal frame field described as before in Lemma \ref{Case1ClassThmDiagonSpacelikekClm10}, the coefficients of the second fundamental form of $M$ and $(s,t)$ a local coordinate system are given as in Lemma \ref{Case1ClassThmDiagonSpacelikekClm120}. \\
Thus, we have 
\begin{equation}\label{AftClm1Eq1SCase10}
e_1= \cos \theta x_s, \quad e_2= \frac{1}{m} x_t.
\end{equation}
By considering these equations, we see that  \eqref{mus0} and \eqref{mut0} become
\begin{equation}
\label{munew0} \mu_s=1, \quad \mu_t=0
\end{equation}
which imply $\mu =s+c$. Up to an appropriated translation on $s$, we may assume $c=0$. Consequently, \eqref{ClassThmDiagonSpacelikekCod1Case11a0}, \eqref{ClassThmDiagonSpacelikekCod1Case11b0} and \eqref{Case1ClassThmDiagonSpacelikekDefm0} turn into, respectively
\begin{eqnarray}
\label{Case1SpacelikeCods0} s^{2}\cos ^{2}\theta \left(h_{22}^3 \right) _{s}+\left(\sin \theta +sh_{22}^3 \right) \left(1+ sh_{22}^3 \sin \theta \right)=0, \\
\label{Case1SpacelikeCods20} s\cos ^{2}\theta(h_{22}^4)_s+\left(1+ sh_{22}^3 \sin \theta \right) h_{22}^4=0, \\
\label{Case1Spacelikems0} s\cos ^{2}\theta m_s-m\left(1+ sh_{22}^3 \sin \theta \right)=0.
\end{eqnarray}
If we put $\displaystyle h_{22}^3=\frac{\sigma}{s}$, then the equation \eqref{Case1SpacelikeCods0} reduces to%
\begin{equation}
s\cos^2 \theta \sigma _{s}=-\sin \theta \left( 1+2\sigma \sin \theta+\sigma ^{2}\right).
\label{TAGGED-3-210}
\end{equation}%
By solving \eqref{TAGGED-3-210} and considering $\theta$ is a non-zero constant, we directly obtain the function $h_{22}^3$ as
\begin{equation}
h_{22}^3(s,t)=\frac{1}{s}\left( \cos \theta \tan \left(\Phi (t)-u(s)\right)-\sin \theta \right),  \label{TAGGED-3-230}
\end{equation}%
where we put $u(s)= \tan \theta \ln s$ and $\Phi$ is a smooth function depending on the parameter $t$. Substituting \eqref{TAGGED-3-230} into \eqref{Case1SpacelikeCods20} and \eqref{Case1Spacelikems0}, respectively, we obtain
\begin{eqnarray}
\label{co20} h_{22}^4(s,t)=s\sec \left(\Phi (t)-u(s)\right) \varrho (t),\\
\label{TAGGED-3-240} m(s,t)=s\cos \left(\Phi (t)-u(s)\right) \varrho (t),
\end{eqnarray}%
for a non-zero smooth function $\varrho$ depending only on $t$. On the other hand, considering the first equation given in \eqref{AftClm1Eq1SCase10} into \eqref{ExpofS1000}, we get 
\begin{equation} \label{newdecompose}
x=s \cos^2 \theta x_s + s \sin \theta e_3.
\end{equation}
By taking consider the first equation given in \eqref{secondfund0} in the Weingarten formula \eqref{MEtomWeingarten} and also as the unit normal vector $e_3$ is parallel, thus the last equation become 
\begin{equation}
s^2 \cos^2 \theta x_{ss}-s \cos^2 \theta x_s +x=0.
\end{equation}
Solving this PDE, we find that the position vector $x$ can be expressed as following
\begin{equation}
x(s,t)=s \cos u c_1(t)+ s\sin u c_2(t) ,
\label{TAGGEDas-3-280}
\end{equation}%
where both $c_1$ and $c_2$ are two vector-valued functions depending only on $t$ in $\mathbb{E}^{4}$ and also $u(s)=\tan \theta \ln s$. Now, we would like show the obtained surface in \eqref{TAGGEDas-3-280} becomes a CSS given in \eqref{x0}. For this reason, we have to show $c_1(t)$ becomes $c{_0}$ being a constant vector. \\
Assume that $c_1^{\prime}\neq 0$. Now, we have from \eqref{TAGGEDas-3-280},
\begin{eqnarray}
\label{xs0} x_s= \frac{1}{\cos \theta} \left[\cos (\tilde{u})c_1 (t)+\sin (\tilde{u})c_2 (t)\right], \\
\label{xt00} x_t= s \cos u c_1^{\prime}(t)+ s \sin u c_2^{\prime}(t),
\end{eqnarray}
where $\tilde{u}=\tilde{u}(s)=\theta+u(s)$. On the other hand, one can consider \eqref{xs0} in $\left\langle x_s,x_s\right\rangle= \sec^2 \theta$ defined in the metric tensor \eqref{ClassThmDiagonSpacelikekDefgEqRESCase10}, so we get
\begin{equation*}
\cos^{2} \tilde{u} \left\langle c_1(t),c_1 (t)\right\rangle +\sin^{2}\tilde{u} \left\langle c_2 (t),c_2 (t)\right\rangle + \sin 2\tilde{u} \left\langle c_1
(t),c_2 (t)\right\rangle =1,
\end{equation*}%
which implies that 
\begin{equation} \label{Eq0}
\left\langle c_1 (t),c_1 (t)\right\rangle =\left\langle c_2(t),c_2 (t)\right\rangle =1, \quad \left\langle c_1 (t),c_2 (t)\right\rangle =0.
\end{equation} 
Indeed, at this point the condition $\left\langle x,x\right\rangle=s^2$ is satisfied. Moreover, from the orthogonality of the expressions given in \eqref{xs0} and \eqref{xt00}, we obtain 
\begin{equation}
\label{Eq1d0} \left\langle c_1 ^{\prime }(t),c_2 (t)\right\rangle =\left\langle c_1 (t),c_2 ^{\prime }(t)\right\rangle =0.
\end{equation}
Finally \eqref{TAGGED-3-240} and \eqref{xt00} considering in $\left\langle x_t,x_t\right\rangle=m^2$, we get 
\begin{subequations} \label{Eq1ALL0}
\begin{eqnarray}
\label{Eq1a0} \left\langle c_1 ^{\prime }(t),c_1 ^{\prime }(t)\right\rangle&=&\varrho ^{2}(t)\cos ^{2}\Phi (t),\\
\label{Eq1b0} \left\langle c_2 ^{\prime }(t),c_2 ^{\prime }(t)\right\rangle &=&\varrho^{2}(t)\sin ^{2}\Phi (t),\\
\label{Eq1c0} \left\langle c_1 ^{\prime }(t),c_2 ^{\prime }(t)\right\rangle &=&\varrho ^{2}(t)\sin \Phi (t)\cos \Phi (t).
\end{eqnarray}
\end{subequations}
Now, let $\left\{X_1,X_2,X_3,X_4\right\}$ be an orthonormal base on $\mathbb E^4$ such that 
\begin{subequations} \label{baseEq1all0}
\begin{eqnarray}
\label{base1a0} X_1 = c_2(t), \quad X_2 = c_1(t), \quad X_3 = \frac{c_2^{\prime}(t)}{\varrho(t)\sin \Phi (t)} ,\\
\label{base1b0} \left\langle X_4,X_4\right\rangle=1 \quad and \quad \left\langle X_4,X_i\right\rangle=0, \quad i=1,2,3.
\end{eqnarray}
\end{subequations}
Here, two vector $c_1(t)$ and $c_2(t)$ satisfy the conditions given in \eqref{Eq0}, \eqref{Eq1d0} and \eqref{Eq1ALL0}. Note that, since $c_1^{\prime}$ is also the another vector in $\mathbb E^4$, so we can write this vector as a linear combination of the orthonormal base given in \eqref{baseEq1all0}. Thus by considering \eqref{Eq0}, \eqref{Eq1d0} and \eqref{Eq1ALL0}, we get 
\begin{equation} \label{ld0}
c_1^{\prime}=\cot \Phi(t) c_2^{\prime}.
\end{equation}
By using \eqref{ld0} into \eqref{xt00}, we get
\begin{equation} \label{xt0}
x_t = s \sec \Phi(t) \cos\Big(\Phi(t)-u(s)\Big)c_1^{\prime}.
\end{equation}
Note that by computing the Weingarten formula \eqref{MEtomWeingarten} for $e_4$, one can get directly  
\begin{equation}
\widetilde{\nabla}_{\partial_s} e_4= 0, \quad \widetilde{\nabla}_{\partial_t} e_4= -h_{22}^4 \partial_t.
\end{equation}
By considering \eqref{co20} and \eqref{xt0} into the last equation, we get $(e_4)_t= -{s^2} \varrho(t) \sec \Phi(t) {c_1}^{\prime}(t)$. Finally, by considering the Schwarz equality $\partial_s \partial_t {e_4} = \partial_t \partial_s {e_4}$, one can get $c_1^{\prime}=0$. But that is contradiction. Thus, $c_1(t)$ becomes $c_0$ being a constant vector and the expression \eqref{TAGGEDas-3-280} becomes a CSS given in \eqref{x0}. 
Consequently, by considering \eqref{x0} and $u(s)= \tan \theta \ln s$ into \eqref{newdecompose} we get \eqref{e30}. One can easily check the equation \eqref{gcrx0} by considering these equations \eqref{AftClm1Eq1SCase10} and \eqref{e30} into \eqref{ExpofS1000}. Thus the necessity is proved. The proof of sufficient condition follows from a direct computation.
\end{proof}

\section{Classifications of Generalized Constant Ratio Surfaces in $\mathbb E^4$} \label{S:Classification}
In this section, we obtain a classification of GCR surfaces. First, we would like to present an example of GCR surface in $\mathbb E^4$.
\begin{example}\label{GCRExampleE4}
Let  $\Pi$ be a hyperplane of $\mathbb E^4$ with  the unit normal $\varphi_0$. 
Assume that $\psi=\psi(t)$ be a curve lying on ${\mathbb S^2(r^2)}={\mathbb S^3(1)}\cap \Pi$ such that $0 < r \leq1$. Consider the surface $M$ given by
\begin{equation} \label{x}
x(s,t)=s\cos u \varphi_0+s \sin u \psi(t)
\end{equation}
for  a smooth function $u=u(s)$. Then,  $x$ can be decompossed
\begin{equation} \label{gcrx}
x(s,t)=s \cos^2 \theta \partial_s+s \sin \theta e_3.
\end{equation}
Here, the unit parallel normal vector $e_3$ is given by
\begin{equation} \label{e300}
e_3 = \sin(\tilde{\theta})\varphi_0-\cos(\tilde{\theta})\psi(t), 
\end{equation}
for a smooth function $\tilde{\theta}=\theta(s)+u(s)$ such that the angle function
 $\theta$ is between $x$ and $e_3$, satifying
\begin{equation} 
 \label{GCRSurfaceCaseIMainTHMEQ1} \tan \theta=s u^{\prime }.
\end{equation}%
Furthermore, a direct computation yields that $\partial_s$ is a principal direction of all shape operators of $M$. Hence, $M$ is GCR.
\end{example}

Let $M$ be a generalized constant ratio surface in $\mathbb E^4$, $x$ its position vector and we put $\left\langle x,x\right\rangle=\mu^2$ for a non-negative function $\mu$. We define a tangent vector field $e_1$ and  a normal vector $e_3$ by
\begin{equation} \label{ExpofS100}
x= \mu\cos \theta e_1+ \mu\sin \theta {e_3},
\end{equation}
for a function $\theta$. Let $e_2$ and ${e_4}$ be a unit tangent vector field and a unit normal vector field, satisfying $\langle e_1,e_2\rangle=0$ and $\left\langle {e_3},{e_4}\right\rangle=0$, respectively. 
Note that $\mu$ satisfies \eqref{mus0} and \eqref{mut0}. We will consider the case $D e_3=0$. 
\begin{rem}
If $\theta$ is constant on an open subset $U$ of $M$, then $U$ is CSS. Since the local classifications of constant slope surfaces given in Theorem \ref{ClassThmDiagonSpacelikek0}, we assume that $\nabla \theta$ does not vanish at any point of $M$.
\end{rem}

By a simple computation considering \eqref{ExpofS100} and the last two equations, we obtain the following lemma.

\begin{lem}\label{Case1ClassThmDiagonSpacelikekClm1}
The Levi-Civita connection $\nabla$ of $M$ is given by
\begin{subequations} \label{CASEISpacelikeLeviCivitaEq1ALL}
\begin{eqnarray}
\label{CASEISpacelikeLeviCivitaEq1a}\nabla _{e_{1}}e_{1}=\nabla _{e_{1}}e_{2}=0, &&
\\\label{CASEILeviCivitaEq1b}
\nabla _{e_{2}}e_{1}= \frac{1+\mu h_{22}^3 \sin \theta}{\mu \cos \theta}e_2, &\quad &\nabla _{e_{2}}e_{2}=-\frac{1+\mu h_{22}^3 \sin \theta}{\mu \cos \theta}e_1. 
\end{eqnarray}
\end{subequations}
and the matrix representations of shape operator $S$ of $M$ with respect to $\{e_1,e_2\}$ are
\begin{align}\label{CASEISpacelikeLeviCivitaEq1}
\begin{split}
S_3=\left(\begin{array}{cc}
-\Big(e_1(\theta)+\frac{\sin \theta}{\mu}\Big)&0\\
0&h_{22}^3
\end{array}\right), \quad
&S_4=\left(\begin{array}{cc}
0&0\\
0&h_{22}^4
\end{array}\right)
\end{split}
\end{align}
for the coefficients of second fundemental form of $M$ satisfying
\begin{subequations} \label{ClassThmDiagonSpacelikekCod1Case1ALL}
\begin{eqnarray} \label{ClassThmDiagonSpacelikekCod1Case11a}
e_1(h_{22}^3)=\frac{1+\mu h_{22}^3 \sin \theta}{\mu \cos \theta}(h_{11}^3-h_{22}^3), &&
\\\label{ClassThmDiagonSpacelikekCod1Case11b}
e_1(h_{22}^4)=-\frac{1+\mu h_{22}^3 \sin \theta}{\mu \cos \theta} h_{22}^4, &&
\\\label{ClassThmDiagonSpacelikekCod1Case11c}
h_{11}^3=-\Big(e_1(\theta)+\frac{\sin \theta}{\mu}\Big), \quad h_{11}^4=0, \quad h_{12}^3=0 \quad h_{12}^4=0.
\end{eqnarray}
\end{subequations}
Furthermore, the angle function $\theta$ satisfies
\begin{equation}\label{CASEISpacelikeLeviCivitaEq1Theta}
e_2(\theta)=0.
\end{equation}
\end{lem}

\begin{proof}
By considering \eqref{ExpofS100}, $\tilde{\nabla}_X x=X$ and the normal vector field $e_3$ being parallel, one can get
\begin{equation}\label{ApplyXExpofS100}
X=X(\mu \cos \theta)e_1+\mu (\cos \theta \nabla_{X}e_1+%
\cos \theta h(e_1,X))-\mu \sin \theta S_{3}X+X(\mu \sin \theta){e_3}
\end{equation}
whenever $X$ is tangent to $M$. By considering \eqref{mus0}, \eqref{ApplyXExpofS100} becomes for $X=e_1$
\begin{eqnarray}
\label{ApplyXExpofS10000Eq2b} h_{11}^3= -\Big(e_1(\theta)+\frac{\sin \theta}{\mu}\Big),\\
\nonumber \nabla_{e_1}e_1=0,\quad \nabla_{e_1}e_2=0,\\
\nonumber h_{11}^4=0.
\end{eqnarray}
While considering \eqref{mut0}, \eqref{ApplyXExpofS100} gives for $X=e_2$
\begin{eqnarray}
\label{thetas} e_2(\theta)=0,\\
\nonumber\nabla_{e_2}e_1= \frac{1+\mu h_{22}^3 \sin \theta}{\mu \cos \theta}e_2, \quad \nabla_{e_2}e_1= -\frac{1+\mu h_{22}^3 \sin \theta}{\mu \cos \theta}e_1, \quad h_{12}^3=0, \quad h_{12}^4=0,
\end{eqnarray}
where $e_2$ is the other principal direction of $M$ corresponding with the principal curvature $h_{22}^3$. Thus, we have \eqref{CASEISpacelikeLeviCivitaEq1ALL} and \eqref{ClassThmDiagonSpacelikekCod1Case11c} and \eqref{CASEISpacelikeLeviCivitaEq1Theta} and the second fundamental form of $M$ becomes 
\begin{equation}
\label{secondfund}
h(e_1,e_1)=-\Big(e_1(\theta)+\frac{\sin \theta}{\mu}\Big)e_3,\quad h(e_1,e_2)=0,\quad\quad h(e_2,e_2)=h_{22}^3e_3+h_{22}^4e_4.
\end{equation}
By considering the Codazzi equation, we obtain \eqref{ClassThmDiagonSpacelikekCod1Case11a} and \eqref{ClassThmDiagonSpacelikekCod1Case11b}.
\end{proof}

Next, we would like to prove the following lemma.
\begin{lem}\label{Case1ClassThmDiagonSpacelikekClm12}
There exists a local coordinate system  $(s,t)$ defined in a neighborhood $\mathcal N_p$ of $p$ such that the induced metric of $M$ is
\begin{equation}\label{ClassThmDiagonSpacelikekDefgEqRESCase1}
g=\frac{1}{{\cos \theta}^2}ds^2+m^2dt^2
\end{equation}
for a smooth function $m$ satisfying
\begin{equation}\label{Case1ClassThmDiagonSpacelikekDefm}
e_1(m)-\frac{1+\mu h_{22}^3 \sin \theta}{\mu \cos \theta}m=0.
\end{equation}
Here $\theta$ is a smooth function. Furthermore, the vector fields $e_1,e_2$ described above become $\displaystyle e_1={\cos \theta}\partial_s$,  $\displaystyle e_2=\frac 1{m}\partial_t$ in $\mathcal N_p$.
\end{lem}

\begin{proof}
We have $\displaystyle [e_1,e_2]=-\frac{1+\mu h_{22}^3 \sin \theta}{\mu \cos \theta}e_2$ because of \eqref{CASEISpacelikeLeviCivitaEq1ALL}. Thus, if $m$ is a non-vanishing smooth function on $M$ satisfying  \eqref{Case1ClassThmDiagonSpacelikekDefm}, then we have $\displaystyle \left[\frac{1}{\cos \theta}e_1,me_2\right]=0$. Therefore, there exists a local coordinate system $(s,t)$ such that $e_1={\cos \theta}\partial_s$ and  $\displaystyle e_2=\frac 1m\partial_t$. Thus, the induced metric of  $M$ is as given in \eqref{ClassThmDiagonSpacelikekDefgEqRESCase1}.
\end{proof}

Now, we are ready to obtain the classification theorem.
\begin{thm}\label{ClassThmDiagonSpacelikek}
Let $x:M\rightarrow \mathbb{E}^{4}$ be a  surface decomposed as in \eqref{ExpofS100} and assume that $e_3$ is parallel. If $M$ is GCR surface, $M$ is locally congruent to a surface given in Example \ref{GCRExampleE4}.
\end{thm}

\begin{proof}
In order to proof the necessary condition, we assume that $M$ is an oriented GCR surface with the isometric immersion $x:M\rightarrow \mathbb E^4$ satisfying $\left\langle x,x\right\rangle=\mu^2$. Since $M$ is a GCR, $x$ can be decomposed as given in \eqref{ExpofS100}. Let $\{e_1,e_2;e_3,e_4\}$ be the local orthonormal frame field described as before in Lemma \ref{Case1ClassThmDiagonSpacelikekClm1}, $h_{11}^3,h_{22}^3$ and $h_{22}^4$ be the principal curvatures of $M$ and $(s,t)$ a local coordinate system given in Lemma \ref{Case1ClassThmDiagonSpacelikekClm12}.
Note that \eqref{thetas} implies $\theta=\theta(s)$. Moreover, we have 
\begin{equation}\label{AftClm1Eq1SCase1}
e_1= \cos \theta x_s.
\end{equation}
Thus \eqref{mus0} and \eqref{mut0} become
\begin{equation}
\label{munew} \mu_s=1, \quad \mu_t=0
\end{equation}
Solving \eqref{munew} one gets $\mu =s+c_{0}$. Up to an appropriated translation on $s$, we may choose $c_{0}=0$.
Consequently \eqref{ClassThmDiagonSpacelikekCod1Case11a}, \eqref{ClassThmDiagonSpacelikekCod1Case11b} and \eqref{Case1ClassThmDiagonSpacelikekDefm} turn into, respectively
\begin{eqnarray}
\label{Case1SpacelikeCods} s^{2}\cos ^{2}\theta \left(h_{22}^3 \right) _{s}+\left( s\cos \theta \theta
^{\prime}+\sin \theta +sh_{22}^3 \right) \left(1+ sh_{22}^3 \sin \theta \right)=0, \\
\label{Case1SpacelikeCods2} s\cos ^{2}\theta(h_{22}^4)_s+\left(1+ sh_{22}^3 \sin \theta \right) h_{22}^4=0, \\
\label{Case1Spacelikems} s\cos ^{2}\theta m_s-m\left(1+ sh_{22}^3 \sin \theta \right)=0.
\end{eqnarray}
If we put 
\begin{equation}
h_{22}^3(s,t)=\frac{\sigma \cos \theta -\sin \theta }{s},
\end{equation}%
then the equation \eqref{Case1SpacelikeCods} reduces to%
\begin{equation}
s\cos \theta \sigma _{s}=-\sin \theta \left( 1+\sigma ^{2}\right).
\label{TAGGED-3-21}
\end{equation}%
Note that the solution of equation \eqref{TAGGED-3-21} is 
\begin{equation*}
\sigma (s,t)=\tan (u(s)-\Phi (t))
\end{equation*}%
where a smooth function $u$ satisfying \eqref{GCRSurfaceCaseIMainTHMEQ1} and
a smooth function $\Phi=\Phi(t)$ on $M$. Hence, the function $h_{22}^3$ becomes 
\begin{equation}
h_{22}^3(s,t)=\frac{1}{s}\left( \cos \theta \tan \left(\Phi (t)-u(s)\right)-\sin \theta \right).  \label{TAGGED-3-23}
\end{equation}%
Considering $\mu =s$ and substituting \eqref{TAGGED-3-23} into \eqref{Case1SpacelikeCods2} and \eqref{Case1Spacelikems}, respectively
\begin{eqnarray}
\label{co2} h_{22}^4(s,t)=s\sec \left(\Phi (t)-u(s)\right) \varrho (t),\\
\label{TAGGED-3-24} m(s,t)=s\cos \left(\Phi (t)-u(s)\right) \varrho (t),
\end{eqnarray}%
for a non-zero smooth function $\varrho$ depending only $t$. Considering these equations in \eqref{CASEISpacelikeLeviCivitaEq1ALL}, one can obtain the Levi-Civita connection on $M$ as regards local coordinates $(s,t)$ on $M$%
\begin{equation*}
\nabla _{\partial s}\partial s=\tan \theta \theta ^{\prime}\partial s, \quad \nabla _{\partial s}\partial t=\frac{m_{s}}{m}\partial t, \quad \nabla _{\partial t}\partial t=\left( -mm_{s}\cos ^{2}\theta \right)\partial s+\frac{m_{t}}{m}\partial t.
\end{equation*}%
The first equation given in \eqref{secondfund} with considering \eqref{AftClm1Eq1SCase1} and the first equation above applying into the Gauss formula, we have%
\begin{equation}
x_{ss}=\tan \theta \theta ^{\prime
}x_s-\left( \mathrm{sec}\theta \theta ^{\prime}+\mathrm{sec}\theta 
\frac{\tan \theta }{s}\right) e_3.  \label{TAGGED-3-25}
\end{equation}%
By considering the decomposition \eqref{ExpofS100} into account and
reordering \eqref{TAGGED-3-25}, we get 
\begin{equation}
s^{2}\cos ^{2}\theta \sin \theta x_{ss}-\left(s^{2}\cos \theta \theta
^{\prime}+s\cos ^{2}\theta \sin \theta \right)x_{s}+\left(\sin \theta
+s\cos \theta \theta ^{\prime }\right)x=0.  \label{TAGGED-3-26}
\end{equation}%
Putting $x=\Psi (s,t)\cdot s$, the previous equation turns into 
\begin{equation}
s^{2}\cos ^{2}\theta \sin \theta \Psi_{ss}+\left(s\cos ^{2}\theta \sin \theta-s^{2}\cos \theta \theta ^{\prime }\right)\Psi _{s}+\sin
^{3}\theta \Psi =0.  \label{TAGGED-3-27}
\end{equation}%
Considering \eqref{GCRSurfaceCaseIMainTHMEQ1}, the PDE \eqref{TAGGED-3-27}
can be rewritten as 
\begin{equation*}
\Psi_{uu}+\Psi =0.
\end{equation*}%
Solving this equation, we find that the position vector $x$ can be expressed
as 
\begin{equation}
x(s,t)=s\left( \cos u\varphi (t)+\sin u\psi (t)\right) ,
\label{TAGGEDas-3-28}
\end{equation}%
where both $\varphi$ and $\psi$ are vector-valued functions depending only on $t$ in $\mathbb{E}^{4}$. Now, we would like show the obtained surface in \eqref{TAGGEDas-3-28} becomes a GCR surface given in \eqref{x}. So, we have to show $\varphi=\varphi_0$ is a constant vector. \\
Assume that $\varphi^{\prime}\neq 0$. Denote, for the sake of simplicity, by $\tilde{\theta}=\tilde{\theta}(s)=\theta(s)+u(s)$. By considering \eqref{TAGGEDas-3-28} in the metric tensor \eqref{ClassThmDiagonSpacelikekDefgEqRESCase1}, one can get
\begin{equation*}
\cos^{2}(\tilde{\theta})\left\langle \varphi
(t),\varphi (t)\right\rangle +\sin^{2}(\tilde{\theta})\left\langle \psi (t),\psi (t)\right\rangle + 
\sin 2(\tilde{\theta})\left\langle \varphi
(t),\psi (t)\right\rangle =1,
\end{equation*}%
which implies that 
\begin{equation} \label{Eq}
\left\langle \varphi (t),\varphi (t)\right\rangle =\left\langle \psi(t),\psi (t)\right\rangle =1, \quad \left\langle \varphi (t),\psi (t)\right\rangle =0.
\end{equation}
Moreover, we have 
\begin{subequations} \label{Eq1ALL}
\begin{eqnarray}
\label{Eq1a} \left\langle \varphi ^{\prime }(t),\varphi ^{\prime }(t)\right\rangle&=&\varrho ^{2}(t)\cos ^{2}\Phi (t),\\
\label{Eq1b} \left\langle \psi ^{\prime }(t),\psi ^{\prime }(t)\right\rangle &=&\varrho^{2}(t)\sin ^{2}\Phi (t),\\
\label{Eq1c} \left\langle \varphi ^{\prime }(t),\psi ^{\prime }(t)\right\rangle &=&\varrho ^{2}(t)\sin \Phi (t)\cos \Phi (t),\\
\label{Eq1d} \left\langle \varphi ^{\prime }(t),\psi (t)\right\rangle &=&\left\langle \varphi (t),\psi ^{\prime }(t)\right\rangle =0.
\end{eqnarray}
\end{subequations}
Let $\left\{X_1,X_2,X_3,X_4\right\}$ be an orthonormal base on $\mathbb E^4$ such that 
\begin{subequations} \label{baseEq1all}
\begin{eqnarray}
\label{base1a} X_1 = \psi(t), \quad X_2 = \varphi(t), \quad X_3 = \frac{\psi^{\prime}(t)}{\varrho(t)\sin \Phi (t)} , \\
\label{base1b} \left\langle X_4,X_4\right\rangle=1 \quad and \quad \left\langle X_4,X_i\right\rangle=0, \quad i=1,2,3.
\end{eqnarray}
\end{subequations}
Since $\varphi^{\prime}$ is also the another vector in $\mathbb E^4$, so we can write as a linear combination of the orthonormal base given in \eqref{baseEq1all}. Thus by considering \eqref{Eq1a} and \eqref{Eq1c}, we get 
\begin{equation} \label{ld}
\varphi^{\prime}=\cot \Phi(t) \psi^{\prime}.
\end{equation}
On the other hand, from \eqref{TAGGEDas-3-28} we have
\begin{equation*}
x_t = s \cos u \varphi^{\prime} +s \sin u \psi^{\prime}.
\end{equation*} 
By using \eqref{ld} into the last equation, we get
\begin{equation} \label{xt}
x_t =\Big( s \cos u  +s \sin u \tan \Phi(t) \Big)\varphi^{\prime}.
\end{equation}
Note that by using the Weingarten formula \eqref{MEtomWeingarten}, one can get directly  
\begin{equation}
\widetilde{\nabla}_{\partial_t} e_4= -h_{22}^4 \partial_t.
\end{equation}
By considering \eqref{co2} and \eqref{xt} into the last equation, we get $\varphi^{\prime}=0$. But that is contradiction. Thus, $\varphi=\varphi_0$ and \eqref{TAGGEDas-3-28} becomes a GCR surface given in \eqref{x}. 
Consequently, by considering \eqref{x} and \eqref{AftClm1Eq1SCase1} into \eqref{ExpofS100} and because of $\mu=s$ we get
\begin{equation} \label{e3}
e_3 = \sin(\tilde{\theta})\varphi_0-\cos(\tilde{\theta})\psi(t)
\end{equation}
where $\tilde{\theta}(s)=\theta(s)+u(s)$ and $\varphi_0$ is a constant vector on $\mathbb S^3(1)$ in $\mathbb E^4$. One can easily check the equation \eqref{gcrx} by considering these equations \eqref{AftClm1Eq1SCase1} and \eqref{e3} into \eqref{ExpofS100}.
The proof of sufficient condition follows from a direct computation.
\end{proof}

\section*{Acknowledgments}
This paper is a part of PhD thesis of the first named author who is supported by The Scientific and Technological Research Council of Turkey (TUBITAK) as a PhD scholar.


\end{document}